\documentclass[preprint]{imsart}

\RequirePackage{amsthm,amsmath}
\RequirePackage[colorlinks,citecolor=blue,urlcolor=blue]{hyperref}

\usepackage{amssymb,mathrsfs}
\usepackage{graphicx}

\startlocaldefs
\numberwithin{equation}{section}
\theoremstyle{plain}
\newtheorem{teo}{Theorem}[section]
\newtheorem{cor}{Corollary}[section]

\theoremstyle{definition}
\newtheorem{defi}{Definition}[section]

\theoremstyle{remark}
\newtheorem*{oss}{Remark}

	\def\d{\sqrt{2}}
	\def\R{\mathbb{R}}
	\def\H{\mathscr{H}}

\newcommand{\fracPart}[1]{[\!|#1|\!]_1}

\endlocaldefs

\begin{document}

\begin{frontmatter}
\title{A canonical form for Gaussian periodic processes}
\runtitle{Gaussian periodic processes}

\begin{aug}
\author{\fnms{Giacomo} \snm{Aletti}
\ead[label=e1]{giacomo.aletti@unimi.it}}
\and
\author{\fnms{Matteo} \snm{Ruffini}
\ead[label=e2]{matteo.ruffini87@gmail.com}  \corref{}}

\runauthor{G.~Aletti and M.~Ruffini}

\affiliation{Universit\`a degli Studi di Milano\thanksmark{m1}
ToolsGroup\thanksmark{m2}}

\address{ADAMSS Center and Dept. of Mathematics\\
Universit\`a degli Studi di Milano\\
\printead{e1}}

\address{ToolGroup and Dept. of Mathematics\\
Universit\`a degli Studi di Milano\\
\printead*{e2}}

\end{aug}

\begin{abstract}
This article provides a representation theorem for a set of Gaussian processes; this theorem allows to build Gaussian processes with arbitrary regularity
and to write them as limit of random trigonometric series.
We show via Karhunen-Loève theorem that this set is isometrically equivalent to $ \ell^2 $. 
We then prove that regularity of trajectory path of anyone of such processes
can be detected just by looking at decrease rate of $ \ell^2 $  sequence associated to him via isometry. 
\end{abstract}

\begin{keyword}[class=AMS]
\kwd[Primary ]{60G15}
\kwd{60G10}
\kwd[; secondary ]{62M20}
\end{keyword}

\begin{keyword}
\kwd{gaussian processes}
\kwd{Karhunen-Lo\`eve's theorem}
\kwd{periodic processes}
\kwd{stationary processes}
\end{keyword}

\end{frontmatter}

\section{Introduction}
The aim of this article is to provide a simple and general method for constructing continuous and periodic Gaussian
processes of arbitrary regularity.
Periodic real processes arise as natural tool in analyzing continuous processes on the circle
(e.g., in image analysis, when processing noises of closed lines). More recently,
they found a great development in the theory of Random Fields on the sphere (see
\cite{Mar-Pec} and the reference therein).

The Brownian Bridge $B_t$ is a well known example of a continuous and periodic Gaussian 
process on $[0,1]$. Although it is not stationary, it becomes stationary if we remove the path-integral
$\int_0^1 B_s ds$, see \cite{Darling}. 
If we model a noise with such a process, 
its supremum may 
be used to test the null hypothesis. Moreover,
this supremum is the limiting distribution of an optimal test statistic for the uniformity of the 
distribution on a circle, \cite{Darling}; in Section~\ref{BB}, we show that this process is strongly related to one
 generated by $B_t$ by randomly choosing the starting point on $[0,1]$.
To generalize this example, we propose an approach which is linked to
Karhunen-Lo\`eve's expansion that gives uncorrelated coefficient.
This expansion is optimal in regression functional studies, as shown in \cite{Cohen}.
In fact, periodic processes are easily decomposed
with Fourier basis, and stationarity will cause the coefficients of the $\sin$ and $\cos$
of the same frequency to be equal. Thus, we underline a natural isometry between this
representation of Gaussian processes and $\ell^2$. 
The asymptotic decay of the Karhunen-Lo\`eve's coefficients will be related to 
the regularity of the paths, as a consequence of Fourier analysis. 
Thus, it will be possible to define a
``periodic fractional Brownian motion'' by choosing
an appropriate asymptotic decay (see also \cite{Adler,Flandrin}) of the coefficients.

For what concerns notations, $s,t,\ldots$ relates to time variables, and will often belong
to $[0,1]$. We denote by $\{x_t\}_{t\in[0,1]},
\{y_t\}_{t\in[0,1]}, \ldots$ stochastic adapted process
defined on a given filtered space $(\Omega, \mathcal{F}, \{\mathcal{F}_t\}_{t\in[0,1]}, \mathbb{P})$,
while $ (X_n)_n,(Y_n)_n,(Z_n)_n,\ldots$ are sequences of random variables. $C(s,t)$ is a 
positive semidefinite function (it will be the correlation function of a stochastic process).
When a process is stationary, its covariance function will often be replaced by 
the associated covariogram function $\tilde{C}(t)$. The sequence $(e_k(t))_{k\in\mathbb{N}}$
denotes a sequence of orthogonal function on $L^2([0,1])$. 
Finally, we denote by $\fracPart{t}$ the fractional part sawtooth function of the real number $t$, 
which is defined by the formula
$\fracPart{t} = t - \operatorname{floor}(t)$.

\section{A canonical form for Gaussian periodical processes}
The first example of signal theory usage in the description of stochastic processes can be found in
\cite{Karhunen}, where is exposed a theorem that allows to represent Gaussian processes as limits of stochastic Fourier series. The classical 
statement of Karhunen-Loève's theorem is the following, as described in  \cite{ash}.
\begin{teo}[Karhunen-Loève]\label{teo:KL}
Let $\{x_t\}_{t\in [a,b]} $,  $a,b<\infty$ , such that
$
E[x_t] = 0,$ $\forall t\in [a,b],
$
and
$
Cov(x_t,x_s) = C(t,s),
$
continuous in both variables. Then 
$$
x_t = \sum_{k=1}^{\infty}Z_ke_k(t),\qquad a\leq t\leq b,
$$
where $e_k$ are the eigenfunction of following integral operator from  $L^2[a,b]$ in itself
$$
f\in L^{2}[a,b]\to g(t) = \int_{a}^{b}{C(t,\tau)f(\tau)d\tau},\,\,\,\,\,\,a\leq t\leq b,
$$
and $e_k$ form an orthonormal bases for the space spanned by eigenfunctions corresponding to nonzero eigenvalues. The $Z_k$ are given by
$$
Z_k = \int_{a}^{b}x_te_k(t)dt
$$
and are orthogonal random variables ($  E(Z_kZ_j) = 0$ for $k\neq j$), with zero mean and variance  $\lambda_k^2$, where $\lambda_k^2$ is
the eigenvalue corresponding to $e_k$.
The series $ \sum_{k=1}^{\infty}Z_ke_k(t) $ converges in mean square to $x_t$, uniformly in $t$,
that is
$$
E\Big((x_t - \sum_{k=1}^{\infty}Z_ke_k(t))^2\Big) \mathop{\longrightarrow}_{n\to\infty}0
$$
uniformly for $t\in[a,b]$.
Moreover if $x_t$ is Gaussian, the $Z_k$ in expansion are real independent Gaussian random variables.
\end{teo}

In this paper we provide a result that allows to build Gaussian processes of arbitrary regularity, proceeding on
the way tracked by Adler (see \cite{Adler}).
The result will be based on Karhunen-Loève's decomposition theorem, and will deal with following set of processes.
\begin{defi}
$ \H$ is the set of real Gaussian stochastic processes $ \{x_t\}_{t\in[0,1]}$ such that they are continuously stationary (so if $ C(s,t) $ is the covariance
function then there exist a real continuous function $ \tilde{C}(s-t) = C(s,t) \,\, \forall s,t,\,\R$), periodical (i.e. $ x_{0} = x_{1} , a.s.$)  with 
$E(x_t) = 0,\,\forall t\in\R  $. 
\end{defi}
The set $ \H$  is a Banach space, when it is equipped with the inner product given by
$$
(\{x_t\}_{t\in[0,1]},\{y_t\}_{t\in[0,1]}) = \int_{0}^{1}{E(x_ty_t)dt}\in\R_+.
$$
\begin{oss}\label{rem:1.1}
 We remark that if  $ \{x_t\}_{t\in[0,1]}\in\H $  and if $ \tilde{C}(s-t) = C(s,t)  $ is its covariogram function, then $ \tilde{C}(t_0)  =  \tilde{C}(t_0+1)$.
\end{oss}

We are going to specialize Karhunen-Loève's decomposition theorem to $\H$, showing that a process is in $\H$ if and only if it can be written as 
limit of a canonical trigonometric random series. From this result we will  show that $\H$ 
may be seen as a Hilbert space, isometrically equivalent to the space of the coefficients $ \ell^2 $; 
via this isometry it will be easy to create Gaussian stationary processes with arbitrary regularity, 
by looking at decreasing speed rate of canonical series 
coefficients. Moreover, this result allows also to detect information about regularity of a process, since it relates it with the regularity of its covariogram function.

\begin{teo}\label{teo:2.1}
 Let $ \{x_t\}_{t\in[0,1]} \in \H $ with covariance $C(s,t) = \tilde{C}(t-s)$; then in mean square, uniformly in $t$, 
$$
x_t = c_0 Y'_0 + \sum_{k=1}^{\infty}c_k(Y_k\sqrt{2}\sin(2k\pi t) + Y'_k\sqrt{2}\cos(2k\pi t))
$$ 
where
$
(Y_n)_n,(Y'_n)_n
$ are two independent sequence of independent standard Gaussian variables,
and
$
(c_k)_k\in \ell^2
$
is such that
$$
c_n^2 = \int_{0}^{1}{\tilde{C}(s)\cos(2n\pi s)ds},\qquad n=0,1,2,\ldots
$$
\end{teo}
\proof
By Mercer Theorem (see, e.g., \cite{ash}) we know that if $ (e_n)_n $ is an orthonormal bases for the space spanned by the 
eigenfunctions corresponding to nonzero eigenvalues of integral operator  
$$
x\in L^{2}[0,1]\to y(t) = \int_{0}^{1}{C(t,\tau)x(\tau)d\tau},\qquad a\leq t\leq b,
$$
then, uniformly, absolutely and in $ L^2[0,1]\times [0,1] $,
$
C(s,t) = \sum_{k=0}^{\infty}{e_k(t)e_k(s)\lambda_k}
$ ,
where $\lambda_k$ is the eigenvalue corresponding to $e_k$.
By Remark~\ref{rem:1.1} we are going to
see that $ (\cos(2n\pi s),\sin(2n\pi s))_n $ are 
eigenfunctions relative to operator  whose kernel is $ C(s,t) $.
In fact, let 
$
a_n= \int_{0}^{1}{\tilde{C}(s)\cos(2n\pi s)ds}
$,
then
$$
\int_{0}^{1}{\cos(2n\pi t)\tilde{C}(t-\tau)dt} =  a_n\cos(2n\pi \tau),
$$
the same relation holding when $\cos$ is replaced by $\sin$.
It follows from Mercer Theorem that  
$$
C(s,t)=a_0+ \sum_{k=1}^{\infty}{2a_k\cos(2k\pi(s- t))}
$$ 
uniformly, absolutely and in $ L^2[0,1]\times [0,1] $ and that $ (a_n)_n\in\ell^1 $,
and hence the sequence $ (c_n)_n $ formed by $ c_n = \sqrt{|a_n|} $ lays in $ \ell^2 $.
From Theorem~\ref{teo:KL} and Theorem~\ref{teo:2.1} we deduce the existence  
of two independent sequence of independent standard Gaussian variables
$
(Y_n)_n,(Y'_n)_n
$ 
such that
in mean square, uniformly in  $ t $
$$
x_t = c_0Y'_0+\sum_{k=1}^{\infty}c_k(Y_k\d\sin(2k\pi t) + Y'_k\d\cos(2k\pi t)).
$$ 
\endproof
\begin{teo}\label{teo:2.2}
Let $
(Y_n)_n,(Y'_n)_n
$ be two independent sequence of independent standard Gaussian variables,
and $(c_k)_k\in \ell^2 $.
Then the sequence
$$
y^{(n)}_t = c_0Y'_0+\sum_{k=1}^{n}c_k(Y_k\d\sin(2k\pi t) + Y'_k\d\cos(2k\pi t))
$$ 
converges in mean square, uniformly in $t$ to $\{y_t\}_{t\in[0,1]}\in\H$.
Moreover if $C(s,t)$ is the $y_t$ covariance function, 
then uniformly, absolutely and in $ L^2[0,1]\times [0,1] $,
$$
C(s,t) = c_0^2+\sum_{k=1}^{\infty}{2c_k^2\cos(2k\pi(s- t))}.
$$

\end{teo}

\proof
First of all we notice that Gaussian process $ y^{(n)}_t $ converges to a periodical 
$ \{y_t\}_{t\in[0,1]} $ in mean square uniformly in $t$, because 
$$
\sup_{t\in[0,1]}E[| y^{(n)}_t - y^{(m)}_t |^{2}] =2 \sum_{k=n}^{m}c_k^2\underbrace{\to}_{m,n} 0.
$$
Let's look at $  \{y_t\}_{t\in[0,1]} $ properties: it is a calculation to show that $ E[y_t] = 0 $ for all $ t $, and 
that 
$$
Cov(y_t,y_s) = c_0^2+2\sum_{k=1}^{\infty}{c_k^2\cos(2k\pi (s-t))},
$$
which is a continuous function, and that
$$
E[y_t^2] = c_0^2+2\sum_{k=1}^{\infty}c_k^2 = 2\|(c_n)_n\|^2-c_0^2.
$$
 Moreover $\{y_t\}_{t\in[0,1]} $ is a Gaussian process, because the 
two sequences $(Y_n)_n$ and$,(Y'_n)_n$ are Gaussians.

\endproof
\begin{cor}\label{cor:isometry}
 Let us consider a couple $ Z=((\bar{Y}_n)_n,(\bar{Y'}_n)_n) $ of independent sequence of
independent standard Gaussian variables. For each $  \{z_t\}_{t\in[0,1]}\in \H $, 
there exists an $  \{x_t\}_{t\in[0,1]}\in \H_{Z}  $ having the same law, where
\begin{multline*}
\H_{Z} = \Big\{ \{x_t\}_{t\in[0,1]} \in \H : \\
x_t =a_0\bar{Y'}_0+\d\sum_{k=1}^{\infty}a_k(\bar{Y}_k\sin(2k\pi t) + \bar{Y'}_k\cos(2k\pi t))
, (a_n)_n\in \ell^2  \Big\}
\end{multline*}
and the limit is in mean square and uniformly in $t$.
\end{cor}
\subsubsection*{Geometry of  $ \H_{Z} $: isometry with $\ell^2$}
For $(a_n)_n\in\ell^2$, define
$$
x_t =a_0\bar{Y'}_0 + \d\sum_{k=1}^{\infty}a_k(\bar{Y}_k\sin(2k\pi t) + \bar{Y'}_k\cos(2k\pi t)). 
$$
From Theorem~\ref{teo:2.1} and Theorem~\ref{teo:2.2} it follows that
$\|x_t\|_{\H} = \sqrt{a_0^2+2\sum_n a_n^2}$, and hence it is naturally defined an isometry
between the representative space $\H_Z$ and $\ell^2$.

\subsubsection*{Regularity of the paths}
We have seen that to each stochastic process in $\H$
can be associated a sequence in $ \ell^2 $. We are now showing how are related the decrease rate of such sequence
with the regularity of the process trajectory path. We first recall a classic regularity theorem.
\begin{teo}[see \cite{Revuz}]\label{teo:2.3}
Let $ \{x_t\}_{t\in[0,1]} $ a real stochastic process such that there exist three positive constants $ \gamma $, $c$ and $ \epsilon $ so that 
$$
E\Big(|x_t-x_s|^{\gamma}\Big) \leq c|t-s|^{1+\epsilon};
$$
so there exists a modification $ \{\tilde{x}_{t}\}_{t\in[0,1]}  $ of $\{x_t\}_{t\in[0,1]}$, such that
$$
E((\sup_{s\neq t}{\frac{|\tilde{x}_t-\tilde{x}_s|}{|t-s|^{\alpha}}})^{\gamma})<\infty
$$
for all $ \alpha\in [0,\frac{\epsilon}{\gamma}) $; in particular the trajectories of $ \{\tilde{x}_{t}\}_{t\in[0,1]}  $
are Holder continuous of order  $ \alpha $.
\end{teo}
 
It is simple to apply this last theorem to processes staying in $ \H $.
It is well known that if $ Y\approx N(0,\sigma^2) $, then 
$
E(|Y|^p) = \sigma^p \frac{2^{\frac{p}{2}}\Gamma\big(\frac{p+1}{2}\big)}{\sqrt{\pi}} ,
$ (see, e.g., \cite{normale}).
From this fact we deduce the following result.
\begin{teo}\label{teo:2.4}
Assume that $ \{x_t\}_{t\in [0,1]}\in\H $ and 
let $C(s,t) = \tilde{C}(s-t)$ be its covariance function. If 
$\tilde{C}$ is Holder continuous of order $ \alpha $, then 
almost all trajectories of  $ \{x_t\}_{t\in[0,1]} $ are Holder continuous of order $ \beta < \frac{\alpha}{2} $.
\end{teo}
\proof
Since
$$
E(|x_{t+h} - x_t|^2) = E(x_t^2 + x_{t+h}^2 - 2x_{t+h}x_{t})  = 2(\tilde{C}(0) - \tilde{C}(h)) \leq M|h|^{\alpha},
$$
we deduce that
$$
E(|x_{t+h} - x_t|^{2p}) = 2^pC_p(\tilde{C}(0) - \tilde{C}(h))^p \leq \tilde{M}|h|^{p\alpha}
$$
so, by Theorem~\ref{teo:2.3}, if $ p\alpha = 1+\epsilon $, for all $p>\frac{1}{\alpha}$, almost all trajectories of
$ \{x_t\}_{t\in[0,1]} $ are Holder continuous of order $\beta $, with $\beta < \frac{p\alpha-1}{2p}$;
but this is true for each $ p > \frac{1}{\alpha} $, we conclude that almost all trajectory path of $ \{x_t\}_{t\in[0,1]} $ are Holder continuous 
of order $\beta < \frac{\alpha}{2}$. 
\endproof
A very useful result for our analysis will be the following one, whose proof may be found in
 \cite{Boas}.
\begin{teo}[Boas' Theorem]
Let $f\in L^1[0,1]$ be a function whose Fourier expansion has only nonnegative cosine terms, 
and let $ (a_n)_n $ be the sequence of its cosine coefficient. Then
$$
f\text{is Holder continuous of order }\alpha \Longleftrightarrow  a_k  = 
O\Big(\frac{1}{k^{\alpha+1}}\Big).
$$
\end{teo}
Boas' Theorem may be used in connection with 
Theorem~\ref{teo:2.1} and Theorem~\ref{teo:2.2} to deduce 
more regularity properties of the processes in $ \H $.
In fact, take $(c_n)_n$ as in Theorem~\ref{teo:2.1} and Theorem~\ref{teo:2.2}.
From Boas' Theorem we have that 
if
$
k^2c_k^2 = O(\frac{1}{k^{1+\alpha}})
$
for $ 0<\alpha\leq 1 $, then $  \partial^2 \tilde{C}  $ is Holder continuous of order $ \alpha $. 
This link between the regularity of $\tilde{C}$ and the paths of $ \{x_t\}_{t\in[0,1]} $
is underlined in the following theorem.
\begin{teo}
With the notations of Theorem~\ref{teo:2.2},
if $c_k^2 = O(\frac{1}{k^{3+\alpha}})$, then
almost all trajectories of  $ \{x_t\}_{t\in[0,1]} $ are Lipshitz continuous, and, as function of $t$, $ \{x_t\}_{t\in[0,1]} $ have a continuous derivative
 $ \{x'_t\}_{t\in[0,1]}$ Holder continuous of order $ \beta < \frac{\alpha}{2} $.
\end{teo}

\proof
It is clear that 
$$
\partial^2 \tilde{C}(\delta) = 2\partial^2 \sum_{k=1}^{\infty}{c^2_k\cos(2k\pi(\delta))} = -2\sum_{k=1}^{\infty}{(2\pi)^2k^2c^2_k\cos(2k\pi(\delta))}
$$
and that $ \partial^2 \tilde{C} $ is Holder continuous of
order $ \alpha $, for some $ 0<\alpha\leq1 $.
Moreover we have that uniformly in $t$ and in mean square
$$
x_t = c_0Y'_0 + \d\sum_{k=1}^{\infty}c_k(Y_k\sin(2k\pi t) + Y'_k\cos(2k\pi t)).
$$
and, from Theorem~\ref{teo:2.1}, there also exist a stochastic process in $ \H $ such that uniformly in $t$ and in mean square
$$
\tilde{x}_t = 2\d\pi\sum_{k=1}^{\infty}kc_k(Y_k\cos(2k\pi t) - Y'_k\sin(2k\pi t)),
$$
which has covariogram function Holder continuous of order $ \alpha $ given by
$$
\tilde{\bar C}(\delta) =  2\sum_{k=1}^{\infty}{(2\pi)^2k^2c^2_k\cos(2k\pi(\delta))}.
$$
If we define
$$
\begin{aligned}
y^{(n)}_t &:= c_0Y'_0 +  \d\sum_{k=1}^{n}c_k(Y_k\sin(2k\pi t) + Y'_k\cos(2k\pi t))
\\
\tilde{y}^{(n)}(t) &:= 2\d\pi\sum_{k=1}^{n}kc_k(Y_k\cos(2k\pi t) - Y'_k\sin(2k\pi t)),
\end{aligned}
$$
than $
y^{(n)}_t = {y}^{(n)}_0 + \int_{0}^{t}\tilde{y}^{(n)}_{\tau}d\tau
$,
a.s. for any $n$, while
for each fixed $t$, in mean square we have
$
  \int_{0}^{t}\tilde{y}^{(n)}_{\tau}d\tau \to \int_{0}^{t}\tilde{x}_{\tau}d\tau.
$
Since
\begin{multline*}
\sqrt{E\Big((x_t - x_0 -  \int_{0}^{t}\tilde{x}_{\tau}d\tau)^2\Big)}  
\leq \sqrt{E\Big((x_t - y^{(n)}_t)^2\Big)} 
\\
+ \sqrt{E\Big((y^{(n)}_0 
+ \int_{0}^{t}\tilde{y}^{(n)}_{\tau}d\tau- x_0 - \int_{0}^{t}\tilde{x}_{\tau}d\tau)^2\Big)}
\mathop{\longrightarrow}_{n\to\infty} 0, 
\end{multline*}
it follows that a.s. 
$
x_t = x_0 + \int_{0}^{t}\tilde{x}_{\tau}d\tau .
$
By Theorem~\ref{teo:2.4} we know that almost all trajectory path of $ \tilde{x}_t$ are Holder continuous of order  $ \beta < \frac{\alpha}{2} $,
and thesis follows.
\endproof
A natural generalization of this result is the following:
\begin{cor}
If, in previous notation, $c_k^2 = O(\frac{1}{k^{1+2m+\alpha}})$ 
 then almost all trajectory path of  $ \{\partial^kx_t\}_{t\in[0,1]} $, with $ k<m $, are Lipschitz continuous,  and admit, as  function of $t$, 
a continuous derivative $ \{\partial^mx_t\}_{t\in[0,1]}$ Holder continuous of order $\beta $, for all $ \beta < \frac{\alpha}{2} $.
\end{cor}

\subsection{The centered Brownian bridge}\label{BB}
Take a Brownian bridge $ \{x_t\}_{t\in[0,1]} $.
This process is Gaussian, periodic but not stationary, since $x_0=x_1 \equiv 0$.
If we randomize the starting point of the process, by shifting the $t$-axis of 
a $[0,1]$-uniform random variable $U$, we obtain the process
\begin{equation}\label{eq:def_sBB}
\hat{x}_t = x_{\fracPart{t-U}}
\end{equation} 
which is expected to belong to $\H$.
A process with the law of $\{\hat{x}_t\}_{t\in[0,1]}$ is called \emph{centered Brownian
bridge}.
Let us recall that 
a Brownian bridge may also be represented as
\[
x_t = \d\sum_{k=1}^\infty W_k \frac{\sin(k \pi t)}{k \pi}.
\]
One may expect that the periodic process
\[
y_t = aY_0 +\d\sum_{k=1}^\infty \frac{b}{k \pi}(Y_k\sin(2k\pi t) + Y'_k\cos(2k\pi t)),
\]
which shares the same asymptotic behavior of the coefficients of $\{x_t\}_{t\in[0,1]}$ is
closely related to $\{\hat{x}_t\}_{t\in[0,1]}$. The next theorem shows this facts.
\begin{teo}
A centered Brownian bridge $\{\hat{x}_t\}_{t\in[0,1]}$ given in \eqref{eq:def_sBB}
belongs to $\H$,
with covariogram function $\tilde{C}(\delta) = \frac{(|\delta|-1/2)^2}{2} + \frac{1}{24}$.
It may be represented as
\[
\hat{x}_t = \frac{1}{\sqrt{12}}
Y_0' +\d\sum_{k=1}^\infty \frac{1}{2k \pi}(Y_k\sin(2k\pi t) + Y'_k\cos(2k\pi t)).
\]
\end{teo}
\begin{proof}
By conditioning on $U$, it is simple to prove that $E(\hat{x}_t)=0$ and
$E(\hat{x}_s\hat{x}_t) = \frac{(|t-s|-1/2)^2}{2} + \frac{1}{24} = \tilde{C}(|t-s|) $.
A straightforward calculation gives
\[
c^2_0 = \int_0^1 \tilde{C}(s) ds = \frac{1}{12}, 
\qquad  c^2_n = \int_0^1 \tilde{C}(s)\cos(2n\pi s) ds = \frac {1}{(2\pi n)^2}, n\geq 1,
\]
and hence the Karhunen-Loève theorem gives
\[
\hat{x}_t = \frac{1}{\sqrt{12}}
Y_0' +\sum_{k=1}^\infty \frac{1}{2k \pi}(Y_k\d\sin(2k\pi t) + Y'_k\d\cos(2k\pi t)).
\]
What remains to prove is that all the $(Y_k,Y'_k)$s are Gaussian. To sketch the proof for
$Y_k$ (the same arguments apply to $Y'_k$),  
we first recall that, conditioned on $U=s$,
\begin{align*}
Z_k := \frac{Y_k}{2\pi k} &= \int_0^1 \hat{x}_t \d\sin(2 k\pi t) dt =
\int_0^1  \d\sum_{n=1}^\infty W_n \frac{\sin(n \pi \fracPart{t-s})}{n \pi} 
\d\sin(2 k\pi t) dt  \\
& = \cos(2k\pi s) \frac{W_{2k}}{2k\pi} + \frac{2 \sin(2k\pi s) }{\pi^2}  
\sum_{n=1,n\neq 2k}^\infty W_n
\frac{(-1)^n-1}{4k^2-n^2}.
\end{align*}
Therefore, the characteristic function $\Phi_{s,Z_k}(t)$ of $Z_k$ conditioned on $U=s$ is
\[
\Phi_{s,Z_k}(t) = e^{-\frac{t^2}{2} \frac{1}{4k^2\pi^2}\cos^2(2k\pi s) }
e^{-\frac{t^2}{2} \frac{4 \sin^2(2k\pi s) }{\pi^4}  
\sum_{n=1,n\neq 2k}^\infty 
\big(\frac{(-1)^n-1}{4k^2-n^2}\big)^2}
\]
and hence the characteristic function $\Phi_{Z_k}(t)$ of $Z_k$ is
\[
\Phi_{Z_k}(t) = \int_0^1 
e^{-\frac{t^2}{2} \frac{1}{4k^2\pi^2}\cos^2(2k\pi s) }
e^{-\frac{t^2}{2} \frac{4 \sin^2(2k\pi s) }{\pi^4}  
\sum_{n=1,n\neq 2k}^\infty 
\big(\frac{(-1)^n-1}{4k^2-n^2}\big)^2}
ds.
\]
Now, since
\[
\frac{4 }{\pi^4}  
\sum_{n=1,n\neq 2k}^\infty 
\Big(\frac{(-1)^n-1}{4k^2-n^2}\Big)^2
=
\frac{1}{k^2\pi^4}\sum_{m=0}^\infty \Big(
\frac{1}{2k+2m+1}+\frac{1}{2k-2m-1}\Big)^2
= \frac{1}{4k^2\pi^2} 
\]
we get
\(
\Phi_{Z_k}(t) = 
e^{-\frac{t^2}{2} \frac{1}{4k^2\pi^2}}
\), which concludes the proof.
\end{proof}

The word ``centered'' in the definition of $\{\hat{x}_t\}_{t\in[0,1]}$ is clearly related
to the randomization of the starting point of the underlying Brownian bridge. 

In fact, we can say more: $\{\hat{x}_t\}_{t\in[0,1]}$ is strongly related to the
$y$-centralization of the Brownian bridge, i.e. to the following process
\begin{equation}\label{eq:def_CCBB}
\check{x}_t = x_t - \int_0^1 x_t \,dt,
\end{equation}
where  $\{x_t\}_{t\in[0,1]}$ is a Brownian bridge.
This last process stays in $\H$,  
and his sup is the limiting distribution of an optimal 
test statistic for the uniformity of the distribution on a circle, see \cite{watson,Darling}.
\begin{cor}
If $\{\hat{x}_t\}_{t\in[0,1]}$ is a centered Brownian Bridge, then it holds
$$
\hat{x}_t = \check{x}_t + Z
$$
where $Z$ is an independent random variable with null expectation and variance $ \frac{1}{12} $, and $\{ \check{x}_t\}_{t\in[0,1]} $ is defined as in \eqref{eq:def_CCBB}. Furthermore
\begin{equation}\label{proc_start}
\check{x}_t = \d\sum_{k=1}^\infty \frac{1}{2k \pi}(Y_k\sin(2k\pi t) + Y'_k\cos(2k\pi t)).
\end{equation}
\end{cor}
\begin{proof}
The covariogram function of $\{\check{x}_t \}_{t\in[0,1]}$ is
$\tilde{C}(\delta) = \frac{(|\delta|-1/2)^2}{2} - \frac{1}{24}$ (see \cite{Darling}).
It is sufficient to calculate the covariance function of $ \check{x}_t + Z $
and to use the Corollary~\ref{cor:isometry} to complete the proof.
\end{proof}

\subsection{A computational parametric model for smoothing}

\begin{figure}[!h]
\begin{center}
\includegraphics[width=13cm, height=10cm]{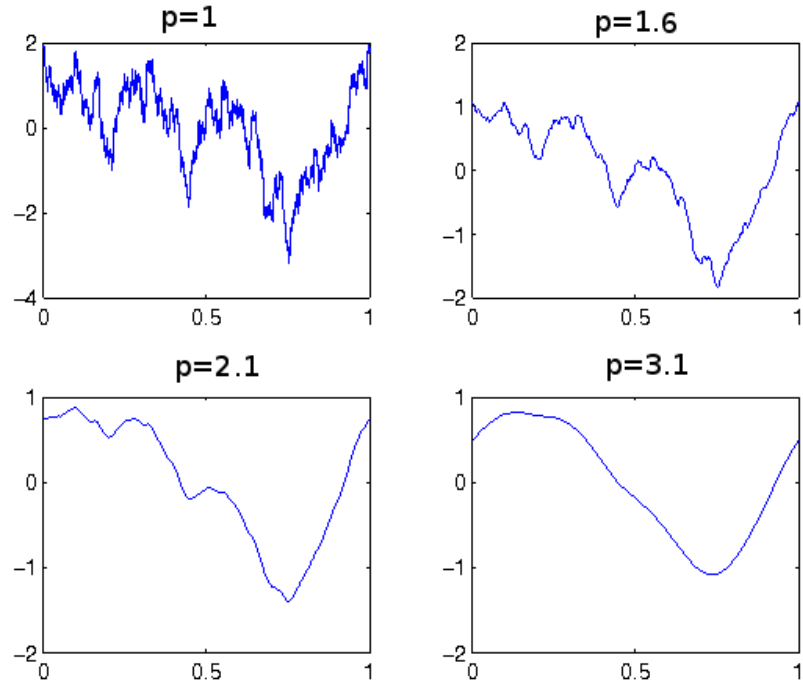}
\caption{A comparison between trajectories of process $x_t$ referred to same event and varying the value of $p$ in $\{1,1.6,2.1,3.1\}$}\label{fig:1}
\end{center}
\end{figure}

Results provided in this paper allows to create a Gaussian parametric family of stationary
and periodic processes of arbitrary
regularity. In fact, let us consider the following family of processes in $\H$ that extends \eqref{proc_start}
\begin{equation}\label{model}
x_t = \sum_{k=1}^{\infty}\frac{a}{k^{p}}(Y_k\sin(2k\pi t) + Y'_k\cos(2k\pi t)).
\end{equation} 
Theorem~\ref{teo:2.4} states that
the paths become more regular as $p$ increases. 
This property is shown in Figure~\ref{fig:1}, which 
suggests how to smooth a process by changing $p$. 

Summing up, model \eqref{model} gives a family of Gaussian processes whose 
trajectories are arbitrarily regular. 
In application, maximum likelihood estimates of $a$ and $p$ is a straightforward consequence of a FFT of the
observed discretized process $\{x_t\}_{t\in[0,1]}$.


\end{document}